\pdfoutput=1

\RequirePackage[l2tabu, orthodox]{nag}

\documentclass[12pt]{amsart}
\usepackage{fullpage,url,amssymb,enumerate,colonequals}
\usepackage[all]{xy} 
\usepackage{mathrsfs} 
\usepackage[dvipsnames]{xcolor}
\usepackage{tikz-cd}
\usepackage{breqn}
\usepackage[normalem]{ulem} 
\usepackage{textcomp}

\usepackage[left=3.1cm, right=3.1cm, top=3cm]{geometry}

\usepackage[OT2,T1]{fontenc}
\usepackage[utf8]{inputenc}
\DeclareSymbolFont{cyrletters}{OT2}{wncyr}{m}{n}
\DeclareMathSymbol{\Sha}{\mathalpha}{cyrletters}{"58}

\numberwithin{equation}{section}

\theoremstyle{definition}
\newtheorem{theorem}{Theorem}[section]
\newtheorem{lemma}[theorem]{Lemma}

\newtheorem*{corollary*}{Corollary}
\newtheorem*{assumption*}{Assumption}
\newtheorem{proposition}[theorem]{Proposition}

\theoremstyle{definition}
\newtheorem{definition}[theorem]{Definition}

\newtheorem*{conjecture*}{Conjecture}
\newtheorem*{theorem*}{Theorem}

\newtheorem*{mainthm*}{Main Theorem}

\theoremstyle{remark}
\newtheorem{remark}[theorem]{Remark}
\newtheorem*{remark*}{Remark}

\makeatletter
\g@addto@macro\bfseries{\boldmath} 
\makeatother

\usepackage{microtype}

\usepackage{hyperref}

\newcommand{\C}{\mathbb{C}} 
\newcommand{\F}{\mathbb{F}}
\newcommand{\G}{\mathbb{G}}

\newcommand{\Q}{\mathbb{Q}}
\newcommand{\R}{\mathbb{R}}

\newcommand{\Z}{\mathbb{Z}}

\newcommand{\xp}{\mathfrak{p}}

\DeclareMathOperator{\divv}{div}

\DeclareMathOperator{\Ext}{Ext}

\DeclareMathOperator{\Gr}{Gr}
\DeclareMathOperator{\Hom}{Hom}

\DeclareMathOperator{\im}{Im}

\DeclareMathOperator{\Jac}{Jac}

\DeclareMathOperator{\Res}{Res}

\DeclareMathOperator{\Nm}{Nm}

\usepackage{mathalfa}
\newcommand{\HHom}{\operatorname{\text{\usefont{U}{BOONDOX-cal}{m}{n}Hom}}} 

\renewcommand{\H}{{\operatorname{H}}}

\newcommand{\dd}{\mathrm{d}}

\usepackage{mathtools}

\newcommand{\norm}[1]{\lVert #1 \rVert}
\newcommand{\abs}[1]{\lvert #1 \rvert}
\newcommand{\inv}{^{-1}}

\newcommand{\ii}{{\texttt i}}

\newcommand{\tpi}{2\pi{\texttt i}}
\newcommand{\tpip}{(2\pi{\texttt i})}

\DeclareMathOperator{\res}{res}
\DeclareMathOperator{\id}{id}

\DeclareMathOperator{\val}{val}

\DeclareMathOperator{\MHS}{MHS}
\DeclareMathOperator{\hgt}{ht}

\DeclareMathOperator{\obst}{obst}

\newcommand{\LMHS}{L}

\newcommand{\cb}{\mathcal{B}}

\newcommand{\cc}{\mathcal{C}}
\newcommand{\cd}{\mathcal{D}}
\newcommand{\ce}{\mathcal{E}}

\newcommand{\co}{\mathcal{O}}

\newcommand{\toi}{\hookrightarrow}
\newcommand{\isoto}{\overset{\sim}{\to}}

\newcommand{\coleq}{\colonequals}

\newcommand{\hatH}{\widehat{H}}

\renewcommand{\Re}{\operatorname{Re}}
\renewcommand{\Im}{\operatorname{Im}}

\title{Heights on curves and limits of Hodge structures}

\author[S.~Bloch]{Spencer Bloch}
\address{Spencer Bloch, University of Chicago, 5801 S Ellis Ave, Chicago, IL 60637, United States}
\email{bloch@math.uchicago.edu }

\author[R.~de~Jong]{Robin de Jong}
\address{Robin de Jong,  Universiteit Leiden, Niels Bohrweg 1, 2333 CA Leiden, The Netherlands}
\email{rdejong@math.leidenuniv.nl}

\author[E.~C.~Sert\"oz]{Emre Can Sert\"oz}
\address{Emre Can Sert\"oz, Institut für Algebraische Geometrie, Leibniz Universit\"at Hannover, 
Welfengarten 1, 30167 Hannover, Germany}
\email{emre@sertoz.com}

\date{\today}

\subjclass[2010]{\href{https://mathscinet.ams.org/msc/msc2010.html?t=11G50}{11G50},
\href{https://mathscinet.ams.org/msc/msc2010.html?t=14D06}{14D06},
\href{https://mathscinet.ams.org/msc/msc2010.html?t=14D07}{14D07},
\href{https://mathscinet.ams.org/msc/msc2010.html?t=14G20}{14G20},
\href{https://mathscinet.ams.org/msc/msc2010.html?t=14H15}{14H15}}


\begin{document}
\begin{abstract} 
  We exhibit a precise connection between  N\'eron--Tate heights on smooth curves and biextension heights of limit mixed Hodge structures associated to smoothing deformations of singular quotient curves. Our approach suggests a new way to compute Beilinson--Bloch heights in higher dimensions.
\end{abstract}
$ $\vspace{-1cm}
\maketitle

\section{Introduction}

Let $X_0$ be an odd-dimensional geometrically integral projective hypersurface of degree at least three defined over a number field. Assume that $X_0$ contains a single nodal singularity. The first named author has recently made the following conjecture.
\begin{conjecture*}
The following two real numbers coincide, up to an element of $\Q \cdot \log |\Q^\times|$: 
\begin{enumerate}[(i)]
  \item the Beilinson--Bloch height~\cite{Bei1987, Bloch1984, Zhang2020} of the difference of the two different rulings on the exceptional quadric of the blow-up of $X_0$ in the node; 
  \item the biextension height~\cite{Hain1990} of the limit mixed Hodge structure determined by any smoothing deformation of $X_0$. \label{item:intro_biext_height}
\end{enumerate}
\end{conjecture*}
\noindent The constraints on $X_0$ are to ensure that the difference of the two rulings is homologically trivial~\cite[see proof of Theorem 2.1]{Schoen1985}. Recently Beilinson announced a proof of this conjecture~\cite{Beil2022}. 

We give a proof of this conjecture for curves and, in particular, give \emph{an explicit formula for the difference} in $\Q \cdot \log |\Q^\times|$. Our approach suggests a new way of computing Beilinson--Bloch heights in practice that avoids the usual elaborate integrals involved in the Archimedean contributions. 

\subsection{Deformations and heights}

Let $K$ be a number field. Let $X_0$ be a geometrically integral projective curve over $K$ containing a single node. For curves, there is no need to assume that $X_0$ is a hypersurface. Our results generalize to the case of multiple nodes; see Section~\ref{sec:smooting_multiple_nodes}. 

We recall that any embedding $\sigma \colon K \toi \C$ together with a first order smoothing deformation of $X_0$ determines a limit mixed Hodge structure~\cite{Schmid1973}. For simplicity, and without loss of generality, we can start with a family of curves $\pi \colon X \to S$ where $S$ is a quasi-projective smooth curve over~$K$ with a base point $0 \in S(K)$, the fibers $X_t = \pi\inv(t)$ are smooth for generic $t$, $X$ is smooth, and we have $X_0 = \pi\inv(0)$. 
Then, each non-zero cotangent vector $\chi \in \Omega_{S,0} \simeq K$ determines a limit mixed Hodge structure $\LMHS_{\chi,\sigma}$ of biextension type.

The real mixed Hodge structure $\Re(\LMHS_{\chi,\sigma})$ obtained by forgetting the integral structure of $L_{\chi,\sigma}$ almost splits into its weight graded pieces. The obstruction space for the splitting of real biextensions is canonically identified with the real numbers~\cite{Hain1990}. The \emph{height} of the biextension $\LMHS_{\chi,\sigma}$ is this obstruction: 
\begin{equation}
\hgt(\LMHS_{\chi,\sigma}) \coleq \obst(\Re(\LMHS_{\chi,\sigma})) \in\R. 
\end{equation}
We define $\hgt(\LMHS_\chi) \coleq \sum_\sigma \hgt(\LMHS_{\chi,\sigma})$, where the sum ranges over all embeddings of $K$ into $\C$.

Given $\lambda \in K^\times$ the height corresponding to the product $\lambda \cdot \chi \in \Omega_{S,0}$ satisfies
\begin{equation}
  \hgt(\LMHS_{\lambda \cdot \chi}) = \hgt(\LMHS_\chi) - \log \abs{\Nm_{K/\Q}(\lambda)},
\end{equation}
where $\Nm_{K/\Q} \colon K \to \Q$ is the norm map.
Therefore, $\hgt(\LMHS_\chi) \bmod \log |\Q^\times|$ is independent of the choice of $\chi$ and, in fact, of the deformation $X/S$, see for instance Theorem~\ref{height_equals_deresonation}. 
This is the height that is referred to in item~\eqref{item:intro_biext_height} above. 

Let $C$ be the normalization of $X_0$ and let $p, q$ be the two points in the preimage of the node of $X_0$ in $C$. For simplicity, assume that $p,q$ are defined over $K$. Let $\hgt(p-q) \in \R_{\ge 0}$ be the (non-normalized) N\'eron--Tate height~\cite{Gross1986,Neron1965} of the linear equivalence class $[p-q]$ in the Jacobian $\Jac(C)$ of $C$.  

Then the conjecture mentioned at the beginning specializes to the following congruence:
  \begin{equation} \label{eq:weak_curves}
    \hgt(p-q) \equiv \hgt(\LMHS_\chi) \mod  \Q \cdot \log |\Q^\times| .
  \end{equation}
Our main theorem is a precise equality in $\R$ that implies the congruence in~\eqref{eq:weak_curves}. 

\subsection{Statement of the main theorem}

Let $\co_K$ be the ring of integers of $K$ and let $\cc/\co_K$ be a proper regular model of $C$. By passing to a quadratic extension of $K$ if necessary, we assume that $p, q$ are defined over $K$. Let $\overline{p}, \overline{q}$ denote the Zariski closures on $\cc$ of $p,q \in C(K)$. 

The Kodaira--Spencer map gives a canonical identification $\Omega_{S,0} \isoto \Omega_{C,p} \otimes \Omega_{C,q}$ of $K$-vector spaces, see Section~\ref{sec:deforming_the_node}. The latter space admits an integral structure obtained from the regular model $\cc$, namely the $\co_K$-lattice 
$\Omega_{\cc/\co_K,\overline{p}} \otimes_{\co_K} \Omega_{\cc/\co_K, \overline{q}} \subset \Omega_{C,p} \otimes_K \Omega_{C,q}$. By transport of structure we obtain an $\co_K$-lattice in $\Omega_{S,0}$ determined by $\cc$. 
This integral structure allows us to define a norm $\norm{\chi} \in \R_{>0}$ for $\chi \in \Omega_{S,0}$, see~\eqref{eq:intro_norm}. 

Let $\Phi$ be a $\Q$-divisor on $\cc$ supported only on the closed fibers of $\cc/\co_K$ such that $\overline{p}-\overline{q}+\Phi$ is of degree zero on every component of every fiber of $\cc$ over $\co_K$. Such a $\Phi$ always exists, see Section~\ref{subsec:disjoint}.

Finally, for any two divisors $\cd,\ce$ on $\cc$ whose supports are disjoint over $C$ we denote by $(\cd \cdot \ce)_{\mathrm{fin}}$ the contribution from the finite places to the Arakelov intersection number of the divisors $\cd, \ce$,  see~\eqref{eq:finite_intersection}.

\begin{mainthm*}  
  There is an equality of real numbers
  \begin{equation}\label{eq:main_intro}
 \hgt(p-q)  = \hgt(\LMHS_\chi) +
\log \norm{\chi} + 2 \,(\overline{p} \cdot \overline{q})_{\mathrm{fin}} - \left( (\overline{p}-\overline{q})\cdot \Phi \right)_{\mathrm{fin}}, 
  \end{equation}
involving the N\'eron--Tate height $\hgt(p-q)$ of the divisor $p-q$ and the biextension height $ \hgt(\LMHS_\chi)$ of the limit mixed Hodge structure $\LMHS_\chi$.
\end{mainthm*}

Note that  the real numbers $\log \norm{\chi}$ and $(\overline{p} \cdot \overline{q})_{\mathrm{fin}}$ lie in $\log |\Q^\times| $. However, the intersection number $\left( (\overline{p}-\overline{q})\cdot \Phi \right)_{\mathrm{fin}}$ lies only in $\Q \cdot \log |\Q^\times| $, since $\Phi$ is a $\Q$-divisor. This explains the congruence condition in~\eqref{eq:weak_curves}.

We emphasize that the equality~\eqref{eq:main_intro} suggests a new way of expressing Beilinson--Bloch heights, not just for curves but in general. A generalization of our proof to higher dimensions is the subject of current research by the authors.

In the companion paper~\cite{BdJS2} we describe how to numerically compute the height of a limit mixed Hodge structure associated to any smoothing of a nodal projective hypersurface of odd dimension, following the ideas in~\cite{Ser2019}. As an application we numerically verify the conjecture stated above for various nodal cubic threefolds.

\subsection{Overview of the paper}  

In Section~\ref{sec:biext_period_heights}, we give an explicit formula to compute the height of a biextension from its mixed periods  (Theorem~\ref{thm:heights_agree}). This section includes an overview of biextensions and their heights, following Hain~\cite{Hain1990}.

In Section~\ref{sec:lmhs}, we express the biextension height of the limit mixed Hodge structure $L_\chi$ as a regularized limit period integral \emph{on the normalization} $C$ of $X_0$ (Theorem~\ref{height_equals_deresonation}).

In Section~\ref{sec:regularized_local_pairing}, we recall N\'eron's local pairings, define \emph{regularized} local N\'eron pairings, and prove a local-to-global formula for regularized N\'eron pairings (Proposition~\ref{prop:local_to_global}). 

In Section~\ref{sec:proof_main} we realize the height of the limit mixed Hodge structure $L_\chi$ as a regularized Archimedean N\'eron pairing evaluated at the divisor $p-q$ (Theorem~\ref{thm:regularized_and_biextension_heights}). We then explicitly compute the regularized non-Archimedean pairing for the divisor $p-q$ (Theorem~\ref{thm:local_nonarch_limit}). Our local-to-global formula for regularized pairings completes the proof of the main theorem. We finish the paper by indicating how the result generalizes to the case of multiple nodes.

\subsection{Acknowledgments}

We thank all members of the International Groupe de Travail on differential equations in Paris for numerous valuable comments and their encouragement and, in particular, Vasily Golyshev, Matt Kerr, and Duco van Straten. 
We thank Greg Pearlstein for helpful discussions and sharing his notes on limits of height pairings with us.
We thank Matthias Schütt for his valuable comments on our paper.
We thank Raymond van Bommel, David Holmes and Steffen M\"uller for helpful discussions. 
We also acknowledge the use of Magma~\cite{Magma} and SageMath~\cite{sagemath} for facilitating experimentation.
We thank Özde Bayer Sertöz and Ali Sinan Sertöz for the pictures in the paper.
The third author gratefully acknowledges support from MPIM Bonn. 
We thank the referee for a careful reading of the manuscript and insightful suggestions.

\section{Biextensions, period matrices and heights} \label{sec:biext_period_heights}

The goal of this section is to express the height of a biextension in a way amenable to calculation. More precisely, we give an explicit formula for the height of a biextension in terms of a \emph{period matrix} of the biextension, see Theorem~\ref{thm:heights_agree}. 

We start this section by recalling the notion of biextensions and their heights. Then we describe the period matrix of a biextension. Then we state and prove our formula for the height of a biextension. We finish this section with a brief discussion of twisted biextensions.

Throughout this section, $H$ denotes a pure integral Hodge structure of weight $-1$. Very little would be lost for our purposes if one were to take $H = \H^1(C,\Z(1))$ for a smooth proper curve $C$. A polarization on $H$ is not needed to define or compute the height.

\subsection{Biextensions and their heights}

We follow Hain's article~\cite{Hain1990} closely. For any integral mixed Hodge structure $V$ we will write $V_\Z \subset V_\R \subset V_\C$ for the underlying $\Z$-lattice and rational, real, and complex vector spaces. The Hodge filtration on $V$ is denoted by $F^\bullet V \subset V_\C$ and the (integral) weight filtration by $W_\bullet V \subset V_\Z$. In general, the weight filtration must be defined on $V_\Q$ in order to construct an abelian category. However, to simplify notation for our particular case we will work with weight filtrations defined on $V_\Z$.

\begin{definition}\label{def:biextension}
  A \emph{biextension} of $H$ is an integral mixed Hodge structure $B$ together with an identification of the weight graded pieces with $\Z(1) \oplus H \oplus \Z$, i.e.,
  \begin{equation}
   \Gr^W B =  \Z(1) \oplus H \oplus \Z.
  \end{equation}
\end{definition}

Following~\cite{Hain1990}, we call $\hatH \coleq \Hom(H,\Z(1))$ the \emph{dual} of $H$. The dual of $\hatH$ is canonically isomorphic to $H$.
If $B$ is a biextension of $H$ then $W_{-1}B$ is an extension of $H$ by $\Z(1)$, while $B/W_{-2}B$ is an extension of $\Z$ by $H$. 

\begin{definition}
  The \emph{Jacobian} $JH$ of $H$ is the complex torus defined by
  \begin{equation}
    JH \coleq H_\C / \left( F^0H + H_\Z \right).
  \end{equation}
\end{definition}
The Jacobian of $H$ parametrizes extensions $0 \to H \to E \to \Z \to 0$, that is we have a canonical identification
  \begin{equation}
    J H = \Ext^1_{\MHS}(\Z,H). 
  \end{equation} 
There is a natural isomorphism between the torus $J \hatH$ and $\Ext^1_{\MHS}(H,\Z(1))$, as we can see by applying $\Hom(\bullet,\Z(1))$ to elements of $J \hatH = \Ext^1_{\MHS}(\Z,\hatH)$.

We see that a biextension $B$ of $H$ determines a point $(B/W_{-2}B,W_{-1}B)$ in $JH \times J \hatH$. However $B$ is not determined by this pair of extensions. Hain~\cite[Lemma 3.2.5]{Hain1990} proves that the moduli of biextensions $\cb(H)$ of $H$ is given by the $\G_m$-torsor associated to 
a twisted Poincar\'e bundle over $JH \times J\hatH$. 

If we forget the integral structure of $H$ and pass to the associated real Hodge structure, then the moduli of biextensions becomes much simpler. Let us write $\Re$ for the forgetful functor from $\Z$- to $\R$-(mixed) Hodge structures.

\begin{proposition}[Proposition 3.2.8 of~\cite{Hain1990}]\label{prop:no_real_exts}
  There are no non-trivial extensions of $\R$ by $\Re(H)$ as real Hodge structures, that is,
  \begin{equation}
    \Ext^1_{\MHS_\R}(\R,\Re(H)) = 0.
  \end{equation}
\end{proposition}
\begin{proof}
  Instead of a quick proof, we will give a longer but somewhat effective proof. Let $E$ be an extension
  \begin{equation}
    0 \to \Re(H) \to E \to \R \to 0.
  \end{equation}
  Choose $s \in E_\R$ and $\eta \in F^0E$ mapping to $1 \in \R$. We have a decomposition $H_\C = F^0H \oplus H_\R$ so we can write $\eta - s = f + r$ with $f \in F^0H$ and $r \in H_\R$. Observe that $s+r = \eta -f \in E_\R \cap F^0H$ splits $E$.
\end{proof}

 Proposition~\ref{prop:no_real_exts} implies that $\Re(B)$ has no moduli with respect to the two extensions $(B/W_{-2}B,W_{-1}B)$. Nevertheless, such ``real biextensions'' of $\Re H$ have moduli, which we denote by $\cb_\R(H)$. 
\begin{definition}
We define a natural map $\varphi \colon \R \to \cb_\R(H)$. Given $c' \in \R$ construct $B = \varphi(c')$ over $B_\R = \R(1) \oplus H_\R \oplus \R$ with the obvious weight filtration and pieces of the Hodge filtration $F^kB_\C$ for $k \neq 0$ coming from the direct sum structure. As for $k=0$, we set $F^0B_\C = F^0H_\C + \C (c',0,1)$. 
\end{definition}
The following statement is a slight modification of Corollary 3.2.9 and Proposition 3.2.13 of~\cite{Hain1990}. 

\begin{proposition}\label{prop:real_biext_moduli}
  The moduli of real biextensions $\cb_\R(H)$ of $\Re H$ is identified with the real line $\R$ via the map $\varphi$ above.
\end{proposition}
\begin{proof}
  We describe the inverse map $\varphi\inv$. Using Proposition~\ref{prop:no_real_exts} we split the real extension $B/W_{-2}B$ identifying it with $\Re H \oplus \R$. As $F^0\R(1) = 0$, there exists a unique $\omega \in F^0 B$ lifting $(0,1) \in B/W_{-2}B = \Re H \oplus \R$. Choose a section $s \in B_\R$ lifting $(0,1) \in B/W_{-2}B$. Split $W_{-1}B = \R(1) \oplus \Re H$ and write $\omega-s = (c'+\tpi \cdot r, 0) \in W_{-1}B$. Then the section $s + \tpi\cdot r$ splits $B_\R$ in such a way that $\omega = (c',0,1)$. 
\end{proof}

As the $\R$-extensions of $\Re H$ and $\Re \hatH$ split, the inverse of $\varphi$ measures the obstruction for the real biextension $\Re B$ to split completely into $\R(1) \oplus \Re H \oplus \R$ as a real mixed Hodge structure.  The obstruction map alluded to in the introduction is $\obst \coleq \varphi\inv$.  

\begin{definition}\label{def:hgt_of_B}
 For a biextension $B$, the \emph{height of $B$} is the real number $\hgt(B) \coleq \obst(\Re(B))$. 
\end{definition}

\subsection{Period matrices and their heights}

We can represent a biextension with a ``period matrix'' which requires a choice of coordinates respecting the weight and Hodge filtrations.

Let $\gamma^0,\dots, \gamma^{2k+1}$ be a basis of $B_\Z$ respecting the weight filtration, i.e.,
\begin{itemize}
  \item $\gamma^0 = \tpi \in \Z(1)$,
  \item $\gamma^1, \, \dots, \, \gamma^{2k} \mod W_{-2}B$ is a basis for $H_\Z$, 
  \item $\gamma^{2k+1} \mod W_{-1}B = 1 \in \Z$.
\end{itemize}
Let $\omega_1,\dots,\omega_{k+1} \in F^0B$ be a basis respecting the weight filtration, i.e.,
\begin{itemize}
  \item $\omega_1, \dots ,\omega_k \mod W_{-2}B \otimes \C$ is a basis for $F^0 H$,
  \item $\omega_{k+1} \mod W_{-1}B \otimes \C = 1 \in F^0 \Z$. 
\end{itemize}

\begin{definition}\label{def:period_matrix} For each $i$, write $\omega_i = \sum_{j} a_{ij} \gamma^j$ with $a_{ij} \in \C$. We call the matrix $P_B = \left( a_{ij} \right)$ a \emph{period matrix of the biextension $B$}.  We can write the period matrix $P_B$ as follows
\begin{equation}\label{eq:PB}
  P_B = 
  \left( 
  \begin{array}{c|ccc|c}
       &  &     &  & \\
    b  &  & P_H &  & 0  \\
       &  &     &  & \\ \hline 
    c  &  & a   &  & 1
  \end{array}
 \right) \, , 
\end{equation}
with $P_H \in \C^{k \times 2k}$, $a \in \C^{1 \times 2k}$, $b \in \C^{k \times 1}$, and $c \in \C$.

 Denote the real and imaginary parts of $P_H$ by $\Re P_H$ and $\Im P_H$ so that $P_H = \Re P_H + \ii \Im P_H$. 
\end{definition}
\begin{definition}\label{def:hgt_of_matrix}
  The \emph{height} of the matrix $P_B$ is given by the formula 
  \begin{equation}\label{eq:hgt_PB}
  \hgt(P_B) = - 2 \pi \left(  \Im c - \Im a \cdot \begin{pmatrix} \Im P_H \\ \Re P_H \end{pmatrix}^{-1} \cdot 
    \begin{pmatrix} \Im b \\ \Re b \end{pmatrix}   \right).  
  \end{equation}
\end{definition}

\subsection{A formula for the height of a biextension}

The main result of this section is as follows.

\begin{theorem}\label{thm:heights_agree} 
  Let $B$ be a biextension. The height $\hgt(B)$ in the sense of Hain \cite{Hain1990} is equal to the height $\hgt(P_B)$, as in Definition~\ref{def:hgt_of_matrix}, of any period matrix $P_B$ of $B$.  
\end{theorem}

\begin{proof}
  We prove that $\hgt(P_B)$ equals $\varphi\inv(\Re B)$ by following the proof of Proposition~\ref{prop:real_biext_moduli} above. We recall the bases $\gamma^j,\, \omega_i$ used to define the period matrix~\eqref{eq:PB}.

  We use the section $s = \gamma^{2k+1} \in B_\R$ and let $\eta = \omega_{k+1} - s \mod W_{-2}B \otimes \C$. Note that $\eta$ lands in $H_\C$. As an $\R$-vector space we have a decomposition $H_\C = H_\R \oplus F^0H$. Write $\eta = r + f$ in this decomposition. As in the proof of Proposition~\ref{prop:no_real_exts}, $\omega_{k+1} - f$ splits $\Re(B)/W_{-2}B = \Re H \oplus \R$.

  There is a unique lift $\tilde f$ of $f$ to $F^0 B$ and $\omega \coleq \omega_{k+1} - \tilde f$ is the unique element in $F^0 B$ mapping to $(0,1) \in \Re(B)/W_{-2}B$. Let $\tilde c \in \C$ be such that $\omega = \tilde c \cdot \gamma^0 \mod \gamma^1,\dots,\gamma^{2k+1}$. It is now clear that $\Re(\tilde c \cdot \gamma^0) = \varphi\inv(\Re B)$.  

  We now express $\tilde f$ and $\tilde c$ using the period matrix $P_B$ and the notation of~\eqref{eq:PB}. The coordinates of $\eta \in H_\C$ with respect to the basis $\gamma^1,\dots,\gamma^{2k} \mod W_{-2}B$ are given by $a \in \C^{1\times 2k}$. We will compute the coordinates of $f \in F^0H$ with respect to $\omega_1,\dots,\omega_{k} \mod W_{-2}B\otimes \C$. The projection $H_\C \to F^0H$ can be factored through $\Im \colon H_\C \to \ii H_\R$ and the inverse of $\Im \colon F^0H \to \ii H_\R$ (as real vector spaces). Using the real bases $\omega_1,\dots,\omega_k, \ii \omega_1, \dots, \ii \omega_k \mod W_{-2}B \otimes \C$ for $F^0H$ and $\ii \gamma_1,\dots, \ii \gamma_{2k}\mod W_{-2}B \otimes \C$ for $\ii H_\R$, the map $F^0 H \to \ii H_\R$ is represented by the matrix
  \begin{equation}
    \begin{pmatrix} \Im P_H \\ \Re P_H \end{pmatrix}
  \end{equation}
acting on the right. It follows that the $\C$-coordinates of $\tilde f$ in $\omega_1,\dots,\omega_k$ are
\begin{equation}
  [\tilde f] \coleq \Im a \cdot \begin{pmatrix} \Im P_H \\ \Re P_H \end{pmatrix}^{-1} \cdot \begin{pmatrix} \id_k \\ \ii \cdot  \id_k \end{pmatrix},
\end{equation}
where $\id_k$ is the $k \times k$ identity matrix. Now $\omega_{k+1}-\tilde f$ has $\gamma_0$ coordinate
\begin{equation}
  \tilde c = c - [\tilde f] \cdot b. 
\end{equation}
Contracting the two products $\id_k \cdot b$ to $b$ and taking imaginary parts gives us the expression~\eqref{eq:hgt_PB}. For this we note that $\Re(\tilde c \gamma^0)=\Re(\tilde c \cdot \tpi) = -2 \pi \Im(\tilde c)$.
\end{proof} 

\subsection{Twisted biextensions}

In geometrical contexts we often encounter Tate twists of biextensions. We make the following definitions.

\begin{definition}\label{def:twisted_biextension}
  An integral mixed Hodge structure $V$ is called a \emph{$k$-twisted biextension} if $V(-k)$ is a biextension $B$, i.e., $V = B(k)$. The height of $V$ is defined to be $\hgt(V) \coleq \hgt(B)$. Further, the notion of a period matrix $P_B$ of a biextension $B$ naturally generalizes to the notion of period matrix $P_V$ of a twisted biextension $V$.
\end{definition}

\begin{remark} 
  \label{rem:twisted} Suppose $V=B(k)$ is a twisted biextension. If $P_B$ is a period matrix of $B$ then $\tpip^{-k} P_B$ is a period matrix of $V$.
\end{remark}

\section{Smoothing a nodal curve}\label{sec:lmhs}

Let $X_0$ be an integral proper complex curve with a single node.  Let $C$ be the normalization of $X_0$ and let $p, q \in C$ be the preimages of the node. 

We put $X_0$ into a proper flat family of curves $X/S$ where $S$ is a connected quasi-projective smooth complex curve with base point $0 \in S$, the total space $X$ is smooth, and the generic fiber  is smooth. 

The choice of a non-zero cotangent vector $\chi \in \Omega_{S,0}$ determines~\cite{Schmid1973} a limit mixed Hodge structure $\LMHS_\chi$ associated to the deformation $X/S$. This limit mixed Hodge structure is a $(-1)$-twisted biextension in the sense of Definition~\ref{def:twisted_biextension}, with central piece $H=\H^1(C,\Z)$. 
Its height is controlled by the element $\chi$ in the sense that
\begin{equation}
  \hgt(\LMHS_{ \lambda \cdot \chi }) = \hgt(\LMHS_\chi) - \log \abs{\lambda} \, ,  \quad \lambda \in \C^\times.
\end{equation}

Our main result in this section, Theorem~\ref{height_equals_deresonation}, gives a formula that expresses the height of our limit mixed Hodge structure in terms of a limit integral \emph{on the curve} $C$ \emph{itself}. 

The Kodaira--Spencer map induces an isomorphism 
\begin{equation} \label{eq:canonical_identification}
\kappa \colon \Omega_{S,0} \isoto \Omega_{C,p} \otimes \Omega_{C,q}
\end{equation}
of one-dimensional $\C$-vector spaces. Let $u$ and $v$ be local holomorphic coordinates around $p$ and $q$, respectively, with $u(p)=v(q)=0$ and such that $\dd u|_p \otimes \dd v|_q = \kappa(\chi)$.

\begin{theorem} \label{height_equals_deresonation}
  Let $\eta \in \H^0(C,\Omega_C(p+q))$ be a differential with simple poles and purely imaginary periods. 
  Assume $\eta$ is scaled so that it has residue $1$ at $p$ and $-1$ at $q$. Then, we have an equality
  \begin{equation}
    \hgt(\LMHS_\chi) =  \lim_{p',q' \to p,q} \Re \left( \int_{q'}^{p'} \eta \right) - \log(\abs{u(p')v(q')}).
  \end{equation}
\end{theorem}

\subsection{The small Kodaira--Spencer map}\label{sec:deforming_the_node}

We recall here the construction of the ``small'' Kodaira--Spencer map~\eqref{eq:canonical_identification}. For our applications, we assume everything in sight is defined over a given subfield $K \subset \C$, including the two points $p,q \in C(K)$, as in the introduction. Then the map $\kappa$ becomes an isomorphism of $K$-vector spaces, as we now explain.

Denote the node of $X_0$ by $x$. 
Locally in the \'etale topology, there are coordinates $\tilde u,\tilde v$ around $x\in X$, and a coordinate $t$ around $0 \in S$ so that the surface $X$ looks like $\tilde u\tilde v = t$ near $x$. 
Replace $S$ by the formal spectrum of $K[[t]]$ and $X$ by the formal spectrum of $K[[\tilde u,\tilde v]]$ so that the morphism $X \to S$ is described by the map $t \mapsto \tilde u \tilde v$. The functions $\tilde u$ and $\tilde v$ lift to local coordinates $u$ and $v$ on $C$ around $p$ and $q$ respectively, possibly after relabeling. Let $\chi \in \Omega_{S,0}$ be such that $\dd t|_0 = \chi$.  We define $\kappa(\chi)$ to be $\dd u|_p \otimes \dd v|_q$. It is straightforward to check that this map does not depend on the choices for $t, \tilde u, \tilde v$. 

\begin{remark}\label{rem:kodaira_spencer}
  To justify the name of the map~\eqref{eq:canonical_identification} we construct it now by alluding to the original Kodaira--Spencer map
\begin{equation}
  \Omega_{S,0}^\vee \to \Ext^1_{\co_{X_0}}(\Omega_{X_0},\co_{X_0})
\end{equation}
associated to the family $X/S$. This construction will not be needed in this paper.

Consider the local-to-global spectral sequence for $\Ext$'s which, in this case, gives
\begin{equation}\label{eq:local_to_global}
  0 \to \H^1(X_0,\HHom(\Omega_{X_0},\co_{X_0})) \to \Ext^1_{\co_{X_0}}(\Omega_{X_0},\co_{X_0}) \to \Ext^1_{\co_{X_0,x}}(\Omega_{X_0,x},\co_{X_0,x}) \to 0.
\end{equation}
The last term in this expression admits an identification 
\begin{equation}
  \Ext^1_{\co_{X_0,x}}(\Omega_{X_0,x},\co_{X_0,x}) \simeq \Omega_{C,p}^\vee \otimes \Omega_{C,q}^\vee,
\end{equation}
canonical up to ordering $p$ and $q$, as explained in~\cite[\S 11, (3.8)]{GACII}. Since our family $X/S$ gives a simple smoothing of the node, the composition of the Kodaira--Spencer map with the projection in~\eqref{eq:local_to_global} gives an isomorphism
\begin{equation}
  \Omega_{S,0}^\vee \isoto \Omega_{C,p}^\vee \otimes \Omega_{C,q}^\vee.
\end{equation}
The inverse dual of this isomorphism coincides with our identification~\eqref{eq:canonical_identification}. 
\end{remark}

\subsection{The limit mixed Hodge structure}

From here on, without loss of generality we take $S$ to be a small open disk around $0$ in $\C$, and write $S^* = S \setminus \{0\}$ for the associated punctured disk. The structure map for the family $X/S$ is denoted by $\pi \colon X \to S$. 

The map $\pi$ induces a variation of pure Hodge structures $V$ of weight $1$ with the $\Z$-local system having fibers $V_{s,\Z} = \H^1(X_s,\Z)$ and the Hodge filtration determined by $F^1 V_s = \H^0(X_s,\Omega_{X_s})$. We get a limit mixed Hodge structure at $0$ once we fix a coordinate function for $S$ at $0$, see~\cite{Schmid1973, Steenbrink1976}. 

In fact, if the differential of two coordinate functions coincide at $0$ then they determine the same limit mixed Hodge structure. 
We will write $\LMHS_\chi$ for the limit mixed Hodge structure induced by any coordinate function $t \colon S \to \C$ satisfying $\dd t|_0 = \chi$. It is easy to see that $\LMHS_\chi$ is a $(-1)$-twisted biextension. See the arguments below.

\subsection{Period matrix of the limit mixed Hodge structure} \label{sec:period_matrix}

We will now give an explicit description of the period matrix of the  twisted biextension $\LMHS_\chi$. We refer to~\cite[\S1.2]{Carlson2017} for a detailed exposition along similar lines.

Let $\omega_1,\dots,\omega_g$ be a basis of holomorphic $1$-forms on $C$ and $\gamma_1,\dots,\gamma_{2g}$ be loops on $C$ giving a basis for $\H_1(C,\Z)$. Choose the loops so that they stay away from $p,q$. Choose a differential of the third kind $\omega_{g+1}$ on $C$ with poles along $p,q$ and $\res_p \omega_{g+1} = 1$, $\res_q \omega_{g+1} = -1$.  Pick a path $\alpha$ from $q$ to $p$ and a small loop $\beta$ oriented counter-clockwise around $p$. 
The basis $\gamma_i$ and the path $\alpha$ descend to loops whose classes form a basis of the first homology of $X_0$. The loop $\beta$ is homologous to zero on $C$ and on $X_0$. The induced loops on $X_0$, as well as their classes, will be denoted by $\gamma_i(0)$, $\alpha(0)$ and $\beta(0)$.

\begin{figure}[!htb]
  \centering
  \begin{minipage}{0.33\textwidth}
    \centering
    \includegraphics[width=0.9\textwidth]{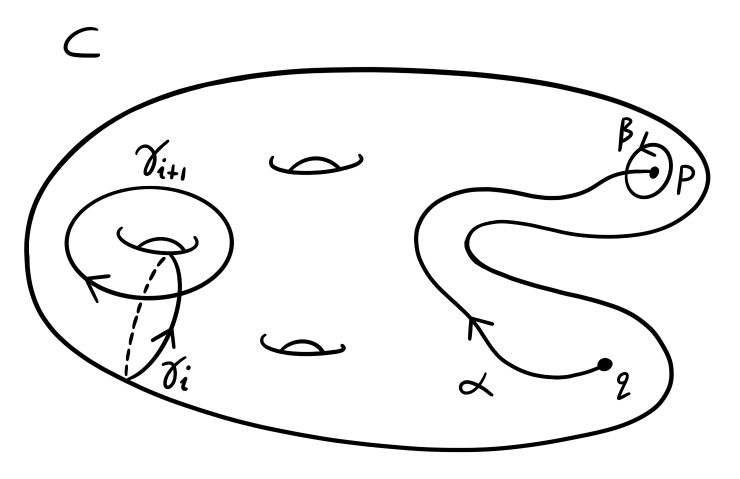}
  \end{minipage}%
  \begin{minipage}{0.33\textwidth}
    \centering
    \includegraphics[width=0.9\textwidth]{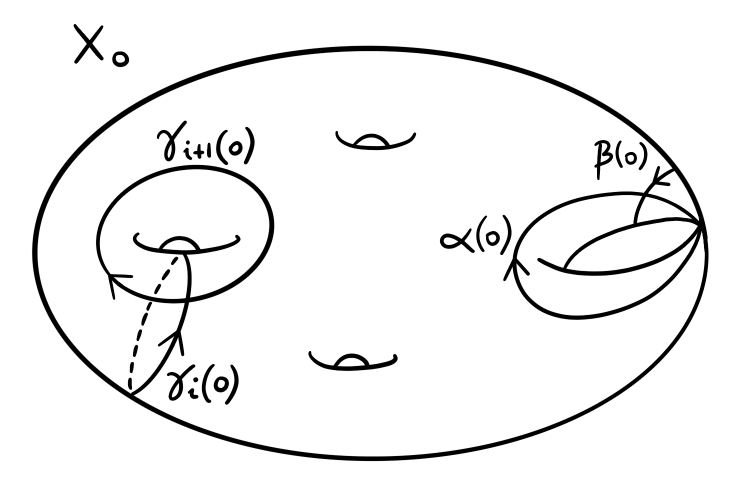}
  \end{minipage}%
  \begin{minipage}{0.33\textwidth}
    \centering
    \includegraphics[width=0.9\textwidth]{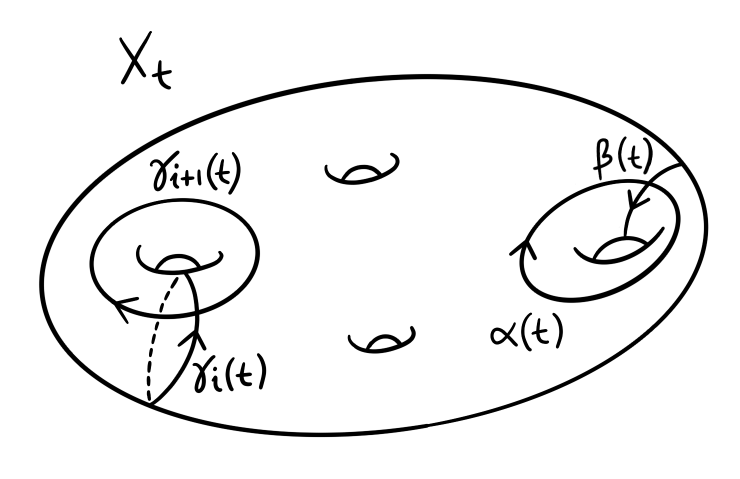}
  \end{minipage}
  \caption{Our designated loops on $C$, $X_0$, and $X_t$.}
\end{figure}

Choose a coordinate function $t$ on $S$ near $0$ with $\dd t|_0 = \chi$. Let $\omega_{X/S}$ be the relative dualizing sheaf of the family $\pi \colon X \to S$. Then $\pi_* \omega_{X/S}$ is a vector bundle of rank $g+1$ over~$S$ with the fibers having the following canonical identifications 
\begin{equation}
  \left. \left(\pi_* \omega_{X/S}\right)\right|_{t} = 
  \begin{cases}
    \H^0(X_t,\Omega_{X_t}) & t \neq 0, \\
    \H^0(C,\Omega_{C}(p+q)) & t = 0.
  \end{cases}
\end{equation}
Pick sections $\mu_1,\dots,\mu_{g+1}$ of $\pi_* \omega_{X/S}$ so that $\mu_i(0) = \omega_i$. In particular, these sections form a basis for $\pi_*\omega_{X/S}$ around $t=0$. 

The integral homology groups $\H_1(X_t,\Z)$ for $t\neq 0$ fit into a local system on $S \setminus \{0\}$. 
The inclusion $X_t \toi X$ gives an exact sequence
\begin{equation}
  0 \to \Z\langle \beta(t) \rangle \to \H_1(X_t,\Z) \to \H_1(X,\Z) \to 0,
\end{equation}
where $\beta(t)$ is the vanishing cycle (unique up to sign). The inclusion $X_0 \toi X$ is a deformation retract, hence $\H_1(X,\Z) \simeq \H_1(X_0,\Z)$. For each $i$ we choose a lift $\gamma_i(t) \in \H_1(X_t,\Z)$ of $\gamma_i(0)$ and a (mildly multivalued) lift $\alpha(t)$ of $\alpha(0)$. We also fix a smooth family of circles representing these loops.  We do the same for $\beta(t)$ so that it specializes to the contractible loop $\beta(0)$, which in turn fixes the orientation for $\beta(t)$.

Let $T \colon \H_1(X_t,\Z) \to \H_1(X_t,\Z)$ be the monodromy operator for going counter-clockwise around $0$. By Picard--Lefschetz theory, $T(\delta) = \delta - \langle \delta,\beta(t) \rangle \beta(t)  $ where $\langle \cdot,\cdot \rangle$ represents the intersection product on homology. By our choice of basis, $T(\gamma_i(t)) = \gamma_i(t)$, $T(\alpha(t)) = \alpha(t) + \beta(t)$, and $T(\beta(t)) = \beta(t)$.

Our choice of orientation on the vanishing cycle $\beta(t)$ and our choice of the residues of $\omega_{g+1}$ imply that
\begin{equation}
  \lim_{t \to 0} \frac{1}{\tpi} \int_{\beta(t)} \mu_{g+1}(t) = \frac{1}{\tpi}\int_\beta \omega_{g+1} = 1.
\end{equation}
Write $\gamma_0(t)=\alpha(t)$ and $\gamma_{2g+1}(t)=\beta(t)$. We will denote the dual cohomology basis by $\gamma^i(t) \in \H^1(X_t,\Z)$ for $i=0,\dots,2g+1$. 

\begin{remark}\label{rem:PL_sign}
Observe that $T(\gamma^i(t)) = \gamma^i(t)$ for all $i < 2g+1$ and $T(\gamma^{2g+1}(t)) = \gamma^{2g+1}(t) - \gamma^0(t)$. Note that the logarithm of monodromy $N=\log T$ evaluates to $T-1$.
\end{remark}

The weight filtration on $\H^1(X_t,\Z)$ is given by~\cite{Schmid1973}
\begin{equation}\label{eq:weight_filtration}
  W_0 = \im (T-1) =  \langle \gamma^0(t)  \rangle,\, W_1 = \ker (T-1) = \langle \gamma^0(t), \dots,\gamma^{2g}(t) \rangle,  \text{ and } W_2 = \H^1(X_t,\Z). 
\end{equation}

As a basis of holomorphic forms on $X_t$ we have
\begin{equation}
F^1 \H^1(X_t,\C) = \H^0(X_t,\Omega^1_{X_t}) = \langle \mu_1(t),\dots,\mu_{g+1}(t) \rangle.
\end{equation}
With respect to these bases, the period matrix of $X_t$ for $t \neq 0$ is equal to
\begin{equation}
  P(t) = \left( \int_{\gamma_j(t)} \mu_i(t) \right)_{\substack{i=1,\dots,g+1 \\ j=0,\dots,2g+1}} \, , 
\end{equation}
see Definition~\ref{def:twisted_biextension}. Let $P_{ij}(t)$ be the $(i,j)$-entry of the matrix $P(t)$ and $e_{ij} \in \C^{(g+1) \times (2g+2)}$ be the matrix with a $1$ at position $(i,j)$ and $0$'s everywhere else. All $P_{ij}(t)$'s are holomorphic on $S$ except for the corner entry $P_{g+1,0}(t)$ which is multi-valued holomorphic at $S \setminus \{0\}$ (single valued if we take a branch cut for $\log(t)$). 

By definition of the limit mixed Hodge structure~\cite{Schmid1973} of the variation $H^1(X_t,\Z)$,  the period matrix of $\LMHS_\chi$ is
\begin{align}
  P_\chi &= \lim_{t \to 0}   P(t) \cdot \exp\left(\frac{\log(t)}{\tpi} N\right) \\
  & = \lim_{t \to 0} \left( P(t) - e_{g+1,0} \frac{\log(t)}{\tpi} \int_{\gamma_{2g+1}(t)} \mu_{g+1}(t) \right)  \\
  &= \lim_{t \to 0} \left( P(t) - e_{g+1,0} \log(t) \right).
\end{align}

We immediately arrive at the following lemma.
\begin{lemma}\label{lem:limit_period_matrix}
  The (biextension) period matrix of $\LMHS_\chi$ (Definition~\ref{def:twisted_biextension}) satisfies 
\begin{equation}\label{eq:limit_period_matrix}
P_\chi=  \left(
  \begin{array}[]{c|ccc|c}
    \int_{\alpha} \omega_1 & \int_{\gamma_1} \omega_1 & \dots  & \int_{\gamma_{2g}} \omega_1 & 0 \\
    \int_{\alpha} \omega_2 & \int_{\gamma_1} \omega_2 & \dots  & \int_{\gamma_{2g}} \omega_2 & 0 \\
    \vdots                 & \vdots                   & \vdots & \vdots                      & \vdots   \\
    \int_{\alpha} \omega_g & \int_{\gamma_1} \omega_g & \dots  & \int_{\gamma_{2g}} \omega_g & 0 \\ \hline
    {\color{red} I_\chi(\alpha,\omega_{g+1})} & \int_{\gamma_1} \omega_{g+1} & \dots  & \int_{\gamma_{2g}} \omega_{g+1} & 2 \pi \ii
  \end{array}
  \right), 
\end{equation}
where 
\begin{eqnarray} \label{eq:Ixi2}
   I_\chi(\alpha,\omega_{g+1}) =  \lim_{t \to 0}  \int_{\gamma_0(t)} \mu_{g+1}(t) - \log(t) \, . 
\end{eqnarray}
\end{lemma}
Note that all integrals in \eqref{eq:limit_period_matrix} are computed on $C$, except for the corner entry $I_\chi(\alpha,\omega_{g+1})$. We will show that this expression too can be computed on $C$. This immediately shows that  the limit mixed Hodge structure is independent of the choice of smoothing of $X_0$, except for the deformation rate of the node.

\begin{remark}
  Our choice of $\Z$-bases for $W_0=\langle \gamma^0(t) \rangle$ and $W_2/W_1 = \langle \gamma^{2g+1}(t) \rangle$ is not canonical. However, the four possible choices of bases induce the same biextension height up to sign. We fixed this sign by insisting for $N=\log T$ to map the basis of $W_2/W_1$ to \emph{minus} the basis of $W_0$. This is equivalent to taking $\gamma^0 \cdot \gamma^{2g+1} = 1$ so that the minus sign is compatible with the Picard--Lefschetz formula.
\end{remark}

\subsection{Proof of Theorem \ref{height_equals_deresonation}} 
We continue with the notation of Section~\ref{sec:period_matrix}.
Theorem~\ref{height_equals_deresonation} follows immediately by combining the three lemmas below. 

\begin{lemma}
  Suppose that $\omega_{g+1}$ has purely imaginary periods. Then 
  \begin{equation} 
    \hgt(\LMHS_\chi) = \Re I_\chi(\alpha,\omega_{g+1}) \, . 
  \end{equation}
\end{lemma}
\begin{proof}
  The period matrix $P_\chi$ of $\LMHS_\chi$ is given in~\eqref{eq:limit_period_matrix}. The biextension $B \coleq \LMHS_\chi(1)$ has period matrix $P_B \coleq \tpip^{-1} P_\chi$. See Remark~\ref{rem:twisted}. Now using the notation in~\eqref{eq:PB} for $P_B$, our assumption on $\omega_{g+1}$ implies that $\Im a=0$. Therefore, the height of $B$, and therefore of $\LMHS_\chi$, is given by
  \begin{equation}
    - 2\pi \Im \frac{I_\chi(\alpha,\omega_{g+1})}{\tpi} = \Re I_\chi(\alpha,\omega_{g+1})
  \end{equation}
  via Definition~\ref{def:hgt_of_matrix}, Theorem~\ref{thm:heights_agree} and Definition~\ref{def:twisted_biextension}.
\end{proof}

As above, we choose coordinates $u$ and $v$ around $p$ and $q$ such that $\dd u|_p \otimes \dd v|_q = \kappa(\chi)$.

\begin{lemma}\label{lem:extrinsic_to_intrinsic}
  If $\omega_{g+1}$ has purely imaginary periods and $r$ is a positive real variable, then
  \begin{equation}
\Re I_\chi(\alpha,\omega_{g+1}) = 
    \lim_{r \to 0^+} \Re \left( \int^{u\inv(\sqrt{r})}_{v\inv(\sqrt{r})} \omega_{g+1} \right) - \log(r).
  \end{equation}
\end{lemma}

\begin{proof} 
  We use the formula~\eqref{eq:Ixi2} to write
  \begin{equation}\label{eq:total_limit}
    I_\chi(\alpha,\omega_{g+1}) = \lim_{r \to 0^+} \int_{\gamma_0(r)} \mu_{g+1}(r) - \log(r).
  \end{equation}
  We will now break the right hand side of~\eqref{eq:total_limit} into three pieces and study them separately. The integration domain will be broken into two pieces: one piece in a neighbourhood of the node and the other in the complement of this neighbourhood. Near the node, we will express the integrand as a sum of an explicit divergent part and a holomorphic part. The resulting three terms appear in the left hand sides of~\eqref{eq:far_from_node}, \eqref{eq:near_node_holo}, \eqref{eq:near_node_divergent}.

  Let $W \subset X$ be an open neighbourhood of the node $x$ with coordinate functions $u,v$ extending the coordinates $u,v$ on the two branches of $X_0$. Fix a small $\varepsilon >0$, let $D_\varepsilon \subset W$ be the rectangular region $\{w \in W \mid \max(\abs{ u(w)},\abs{ v(w)}) < \varepsilon\}$. Analogously, we define $D_{\sqrt{r}}$.

  Take a smoothly deforming family of paths $\psi(r)$ representing the homology class $\gamma_0(r)$.
  Furthermore, we may assume that $\psi_\varepsilon(r) \coleq \psi(r) \cap D_\varepsilon$ is the real line segment with coordinates $ u  v = r$ in $\R^2 \cap D_\varepsilon$. Note that $\psi_\varepsilon(0)$ becomes the union of the two real axes in $D_\varepsilon$.  Taking $r \ll \varepsilon$, the orientation of $\psi_\varepsilon(r)$ is so that it points from $(\varepsilon,r/\varepsilon)$ to $(r/\varepsilon,\varepsilon)$.

  \begin{figure}[!h]
    \centering
    \includegraphics[width=0.5\textwidth]{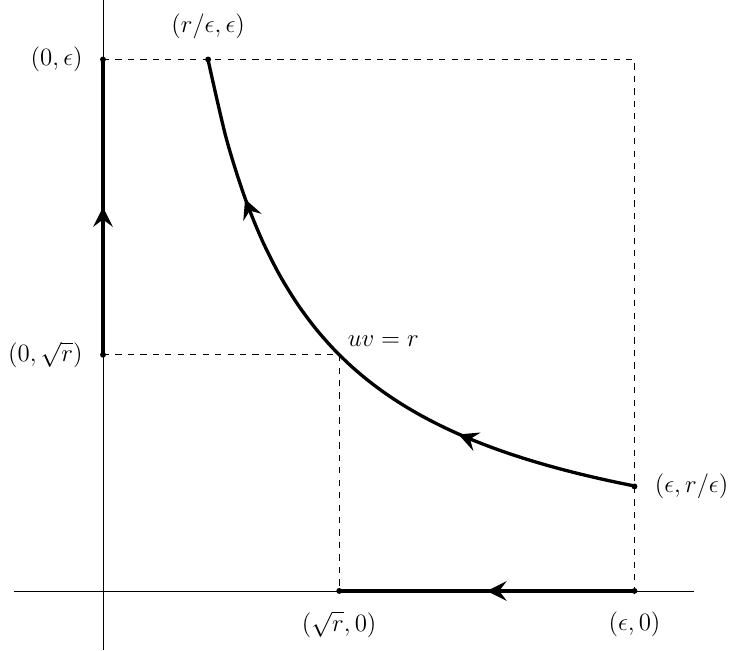}
    \caption{A depiction of the paths $\psi_\varepsilon(r)$ and $\psi_{\sqrt{r},\varepsilon}(0)$.}
  \end{figure}

  With $D_\varepsilon^c$ the complement of $D_\varepsilon$ in $X$, let $\psi_\varepsilon^c(r) \coleq \psi(r) \cap D_\varepsilon^c$. Since $\mu_{g+1}(r)$ is holomorphic on the complement $D_\varepsilon^c$, and $\psi(r)$ varies smoothly, the following equality is immediate
  \begin{equation}\label{eq:far_from_node}
    \lim_{r \to 0^+} \int_{\psi_\varepsilon^c(r)} \mu_{g+1}(r) =  \int_{\psi_\varepsilon^c(0)} \mu_{g+1}(0).
  \end{equation}

  For $r > 0$ and on $X_r^\varepsilon \coleq X_r \cap D_\varepsilon \subset W$ we can construct the following Poincar\'e--Griffiths residue 
  \begin{equation}
    \nu(r) = \res_{X_r^\varepsilon} \frac{\dd u \dd v}{uv - r}.
  \end{equation}
  We can write this form alternatively as
  \begin{equation}\label{eq:mur2}
    \nu(r) = \left.\frac{\dd u}{u}\right|_{X_r^\varepsilon} = - \left.\frac{\dd v}{v}\right|_{X_r^\varepsilon}.
  \end{equation}
  The limit of $\nu(r)$ on $X_0^\varepsilon$ will be defined by the formula~\eqref{eq:mur2}, but using the first expression $\dd u / u$ when $v=0$ and the second expression $-\dd v /v$ when $u=0$.

  Observe that the difference $\mu_{g+1}(r) - \nu(r)$ is holomorphic on $X_r^\varepsilon$ for all $r \ge 0$. 
This gives
\begin{equation}\label{eq:near_node_holo}
    \lim_{r\to 0^+} \int_{\psi_\varepsilon(r)} \mu_{g+1}(r) - \nu(r) = 
    \int_{\psi_\varepsilon(0)} \mu_{g+1}(0) - \nu(0). 
  \end{equation}
Let $\psi_{\sqrt{r},\varepsilon}(0) = \psi(0) \cap D_{\sqrt{r}}^c \cap D_\varepsilon$. Then, the last integral clearly satisfies 
\begin{equation}
  \int_{\psi_\varepsilon(0)} \mu_{g+1}(0) - \nu(0) = \lim_{r \to 0^+} \int_{\psi_{\sqrt{r},\varepsilon}(0)} \mu_{g+1}(0) - \nu(0). 
\end{equation}

  To conclude the proof, it remains to show that the following two limits agree
  \begin{equation}\label{eq:near_node_divergent}
    \lim_{r \to 0^+} \int_{\psi_\varepsilon(r)} \nu(r) - \log(r) = \lim_{r \to 0^+} \int_{\psi_{\sqrt{r},\varepsilon}(0)} \nu(0) - \log(r).
  \end{equation}

  Recall $\psi_\varepsilon(r)$ begins at $(\varepsilon,r/\varepsilon)$ and ends at $(r/\varepsilon,\varepsilon)$ so
  \begin{equation}
    \int_{\psi_\varepsilon(r)} \nu(r) = \int_{\psi_\varepsilon(r)} \frac{\dd u}{u} = \log(r) - 2 \log(\varepsilon).
  \end{equation}

  The path $\psi_{\sqrt{r},\varepsilon}(0)$ has two components along the two axes and this gives
  \begin{equation}
    \int_{\psi_{\sqrt{r},\varepsilon}(0)} \nu(0) = \int^{\sqrt{r}}_\varepsilon \frac{\dd u}{u} + \int^{\varepsilon}_{\sqrt{r}} - \frac{\dd v}{v} = \log(r) - 2 \log(\varepsilon). \qedhere
  \end{equation}
\end{proof}

\begin{remark}
  We note here the similarity of the previous proof with that of~\cite[III.B.11]{GGK2008}. The referenced result may be interpreted as a special case of our main theorem in the case of genus $0$ curves. 
\end{remark}

The proof of Theorem~\ref{height_equals_deresonation} is completed once we show the path independence of the limit appearing in Lemma~\ref{lem:extrinsic_to_intrinsic}. 

\begin{lemma}\label{lem:well_defined} 
  Suppose $\omega_{g+1}$ has purely imaginary periods. Then the following limit  
  \begin{equation}\label{eq:well_defined}
    \lim_{p',q' \to p,q} \Re \left( \int_{q'}^{p'} \omega_{g+1} \right) - \log(\abs{u(p')v(q')}).
  \end{equation}
  is independent of the way the points $p',q'$ approach $p,q$ in~$C\setminus \{p,q\}$.
\end{lemma}
\begin{proof}
  This is straightforward. But we can also appeal to equations~\eqref{eq:local_pairing_at_infty} and~\eqref{eq:pq_limit} below which imply the convergence of this integral through Proposition~\ref{prop:regularization}.  
\end{proof}

\section{Regularized local N\'eron pairings}\label{sec:regularized_local_pairing}

We begin by recalling the standard local N\'eron pairing for disjoint divisors of degree zero on a curve over a number field. We then give our definition of regularized local pairing for divisors with common support by removing divergent terms from a limit of local intersections, see Definition~\ref{def:limit_local_pairing}. 

Our definition is comparable to the extension of the N\'eron pairing by Gross~\cite[Section~5]{Gross1986}, see Remark~\ref{rem:compare_to_Gross}. We use this comparison to prove a local-to-global formula for our regularized local pairings. The regularized Archimedean pairing could also be compared to the limit biextension heights defined by Brosnan and Pearlstein~\cite{Brosnan2019}, see Remark~\ref{rem:compare_to_BP}.

\subsection{Local N\'eron pairings in the case of disjoint support} \label{subsec:disjoint}

The main reference for this section is~\cite{Gross1986}.
Let $K$ be a number field, and let $C$ be a smooth projective geometrically connected curve over $K$. Let $D,E$ be $K$-divisors of degree $0$ on $C$ with disjoint support. Then the N\'eron pairing of $D$ and $E$ decomposes into local components as follows
\begin{equation} \label{local_decomposition}
  \langle D,E \rangle = \sum_\nu \langle D,E \rangle_\nu \, , 
\end{equation}
where $\nu$ runs through the places of $K$. 

The N\'eron pairing is invariant under linear equivalence and descends to a positive definite bilinear form on $\Jac(C)(K) \otimes_\Z \R$. The associated quadratic form is the (non-normalized) N\'eron--Tate height $\hgt(D)$ of divisor classes associated to \emph{twice} the canonical principal polarization of $\Jac(C)$, see~\cite[Section~4]{Gross1986}. 

\subsubsection{Local fields and absolute values}

For each place $\nu$ of $K$, we denote by $\C_\nu $ the completion of an algebraic closure of $K_\nu$, endowed with the $\nu$-adic topology.  The pairing $\langle D,E \rangle_\nu$ extends to $\C_\nu$-divisors with disjoint support, and this pairing is continuous on both sides provided supports remain disjoint. 

For each place $\nu$ of $K$, we fix an absolute value $|\cdot|_\nu$ on $K_\nu$ by demanding that $|\pi|_\nu^{-1} = \Nm(\xp)$ if $\nu$ is non-Archimedean, given by a maximal ideal $\xp$ of norm $\Nm(\xp)$, and with $\pi$ a generator of the maximal ideal of $\co_{K,\xp}$.  When $\nu$ is a real embedding we set $|\cdot|_\nu$ to be the usual absolute value on $K_\nu=\R$; if $\nu$ is a pair of complex conjugate embeddings we set $|\cdot|_\nu$ to be the square of the usual absolute value on $K_\nu=\C$.  The absolute value $|\cdot|_\nu$ extends uniquely as an absolute value on $\C_\nu$.

\subsubsection{Archimedean pairing}
If $\nu$ is an Archimedean place of $K$, the component $\langle D,E \rangle_\nu$ is defined as follows. By abuse of notation continue to denote by $C,D,E$ their pullback from $K$ to $\C$ via $\nu$. Let $\eta_E$ be the unique differential with simple poles on $C(\C)$, with purely imaginary periods, and with residual divisor equal to $E$, that is,
\begin{equation}
E = \Res(\eta_E) \coleq \sum_{r \in C(\C)} \res_r(\eta_E) \cdot r.
\end{equation}
Let $\gamma_{D}$ be any $1$-chain on $C(\C)$ such that $\partial \gamma_D = D$. Set $\epsilon_\nu = 1$ if $\nu$ is a real embedding, and $\epsilon_\nu=2$ if $\nu$ is a pair of complex conjugate embeddings. Then 
\begin{equation}\label{eq:local_pairing_at_infty}
  \langle D,E \rangle_\nu \coleq \epsilon_\nu \Re \int_{\gamma_D} \eta_E.
\end{equation}

\subsubsection{Non-Archimedean pairing}\label{sec:non-arch_pairing}
If $\nu$ is a non-Archimedean place, that is, a maximal ideal $\xp$ of the ring of integers $\co_K$ of $K$, the component $\langle D,E \rangle_\xp$ is defined as follows.
Let $\cc_\xp$ be a proper regular model of $C$ over $\co_{K,\xp}$ and let $\cd,\ce$ be $\Q$-divisors on $\cc_\xp$ extending $D,E$ so that one of $\cd$, $\ce$  restricts to a divisor of degree $0$ on every component of $\cc_\xp|_{\overline{\F}_\xp}$.  By \cite[Theorem~9.1.23]{Liu2006},  the intersection pairing on the fiber $\cc_\xp|_{\overline{\F}_\xp}$ is negative semi-definite, with radical spanned by the fiber itself. From this it follows that such  $\Q$-divisors $\cd, \ce$ exist. Then
\begin{equation} \label{def:local_non_arch_pairing}
  \langle D,E \rangle_\xp \coleq -\iota_\xp(\cd,\ce) \log \Nm(\xp) \, , 
\end{equation}
where $\iota_\xp$ refers to the intersection multiplicity on the regular model $\cc_\xp$ and where $\Nm(\xp)$ is the norm of the maximal ideal $\xp$.

\subsection{The definition of regularized local pairings}

We note that when $D$ and $E$ have common support on $C$, both the Archimedean and the non-Archimedean local intersection pairings (\ref{eq:local_pairing_at_infty}) and (\ref{def:local_non_arch_pairing}) are divergent. However, their global N\'eron pairing still can be defined, using that the global pairing is invariant under linear equivalence. 

Gross~\cite[Section~5]{Gross1986} proposes fixing tangent vectors at the common support of $D$ and $E$ to extend the above local pairing decomposition to the case with common supports. We follow his lead but using a dual notation.

For any divisor $D$ on $C$ let $\abs{D}$ stand for the support of $D$. Let 
\begin{equation}
  Z=\abs{D} \cap \abs{E} = \{p_1,\dots,p_m\}
\end{equation}
be the reduced common support of $D$ and $E$. We will assume for simplicity that the points in $Z$ are $K$-rational points of $C$. Let $a_i$ be the multiplicity of $p_i$ in~$D$ and $b_i$ the multiplicity of $p_i$ in $E$. 

Pick an element $\xi$ in the image of the following map
\begin{equation}\label{eq:image_xi}
  \bigoplus_{i=1}^m \Omega_{C,p_i} \to \bigotimes_{i=1}^m \Omega_{C,p_i}^{\otimes a_ib_i} : (\alpha_1, \dots, \alpha_m) \mapsto \alpha_1^{\otimes a_1b_1} \otimes \dots \otimes \alpha_m^{\otimes a_mb_m},
\end{equation}
where the tensor product is taken over $K$. 

Choose for each $i$, a rational function $u_i \in K(C)$ vanishing on $p_i$ to order $1$ so that $\xi = \bigotimes_{i=1}^m (\dd u_i|_{p_i})^{\otimes a_ib_i}$. Note that that $u_i$ is a local coordinate around $p_i$ at every place of $K$. We will say that the tuple of coordinates $(u_1,\dots,u_m)$ determines $\xi$.

We write $D' \to D$ to mean the points $p_i$ in $D$ are perturbed to give points $p_i'$ in $D'$ and then made to approach $p_i$ in $C(\C_\nu)$. 
\begin{definition}\label{def:limit_local_pairing} 
The \emph{regularization} $\langle D,E \rangle_{\xi,\nu}$ of the $\nu$-local N\'eron pairing is defined as 
\begin{equation}\label{eq:limit_local_pairing}
    \langle D,E \rangle_{\xi,\nu} = \lim_{D' \to D} \left( \langle D',E \rangle_\nu -  \sum_{i=1}^m a_ib_i\log\abs{u_i(p'_i)}_\nu \right).
  \end{equation}
\end{definition}

We will show in Proposition~\ref{prop:regularization} that the limit exists, is well defined, and is compatible with Gross' definition~\cite[Section~5]{Gross1986}. As a consequence we will see that it satisfies a local-to-global formula computing the global pairing $\langle D, E \rangle$.

\begin{remark}\label{rem:compare_to_BP} 
  Over an Archimedean place, if we correctly arrange the deformation $D' \to D$ to be over a disk and form the family of Hain biextensions associated to each of the pair of disjoint cycles $D'$ and $E$~\cite[Section 3.3]{Hain1990}, then the limit of the biextension heights as defined by Brosnan and Pearlstein~\cite[Definition 70]{Brosnan2019} will equal our regularized pairing in~\eqref{eq:limit_local_pairing}.
\end{remark}

\subsection{Local-to-global formula for regularized local pairings}

Let $f \in K(C)$ be a rational function. Let $(u_1,\dots,u_m)$ be any tuple of coordinates determining $\xi$. 

\begin{definition}
   We say $f$ is \emph{compatible with $(D,E,\xi)$} if $D$ and $E + \divv(f)$ have disjoint support and, for each $i$, the formal local expansion of $f \in \Q[[u_i]]$ around $p_i$ has leading coefficient $1$, and $f$ evaluates to $1$ at every other point in $D$.
\end{definition}

\begin{remark} It is a standard result that an $f$ compatible with $(D,E,\xi)$ always exists.
\end{remark}

\begin{proposition} \label{prop:regularization}
Pick any $f$ compatible with $(D,E,\xi)$. Then the regularized local pairings are well defined and they satisfy the following identity
  \begin{equation}
    \langle D,E \rangle_{\xi,\nu} = \langle D, E + \divv(f) \rangle_{\nu}.
  \end{equation}
\end{proposition}
For the proof we need the following definition.
\begin{definition}
Let $F/K$ be a field extension. When $Z=\sum c_i z_i$ with $c_i \in \Z$  is a divisor on $C$ with $z_i \in C(F)$ and $f \in F(C)$ a function such that $\divv(f)$ and $Z$ have disjoint support we write
\begin{equation} \label{define_function_eval}
 f(Z) \coleq \prod_i f(z_i)^{c_i} \in F^\times \, . 
\end{equation}
\end{definition}

\begin{proof}[Proof of Proposition~\ref{prop:regularization}] 
  Pick any $f$ compatible with $(D,E,\xi)$. Let $\nu$ be a place of $K$.
  The local N\'eron pairings are continuous in both arguments and hence we have 
  \begin{equation}
   \langle D, E + \divv(f) \rangle_{\nu}    = \lim_{D' \to D} \langle D', E + \divv(f) \rangle_{\nu}.
  \end{equation}
  We expand out the term in the limit to get
  \begin{equation}
    \langle D', E + \divv(f) \rangle_{\nu} = \langle D',E \rangle_\nu + \log\abs{f(D')}_\nu
  \end{equation}
  where $f(D')$ is as in~\eqref{define_function_eval}.  Since $f$ is compatible with $(D,E,\xi)$ we have $f = u_i^{-b_i}(1 + O(u_i))$ near $p_i$ and $f$ evaluates to $1$ on other points on $D$. Therefore, in the limit, we have the following equality
  \begin{equation}
    \lim_{D' \to D}  \langle D',E \rangle_\nu + \log\abs{f(D')}_\nu = \lim_{D' \to D} \langle D',E \rangle_\nu - \sum_{i=1}^m a_ib_i \log\abs{u_i(p_i')}_{\nu}. \qedhere
  \end{equation}
\end{proof}

\begin{remark}\label{rem:compare_to_Gross}
 We see from Proposition~\ref{prop:regularization} that if for each $i$ we fix the tangent vector $\partial_{u_i}|_{p_i}$ dual to $\dd u_i |_{p_i}$,  our definition of the regularized local pairing $ \langle D,E \rangle_{\xi,\nu}$ is a special case of  Gross' definition in~\cite[Section~5]{Gross1986} based on the tangent vectors $\partial_{u_i}|_{p_i}$. We have arranged matters so that the correction term $\log\abs{f[D]}_\nu$ introduced in \cite[Section~5]{Gross1986} is zero for each $\nu$. 
\end{remark}

The original local-to-global formula~\eqref{local_decomposition} is defined for a pair of disjoint divisors. We now extend this formula using our regularized local pairings to an arbitrary pair of divisors but with the choice of $\xi$ as in~\eqref{eq:image_xi}.

\begin{proposition}\label{prop:local_to_global}
  With $D,E,\xi$ as above, we have a local-to-global formula
  \begin{equation}
    \langle D,E \rangle = \sum_\nu \langle D,E \rangle_{\xi,\nu}.
  \end{equation}
\end{proposition}
\begin{proof} Pick any $f$ compatible with $(D,E,\xi)$. 
  As the global pairing is invariant under linear equivalence we have
  \begin{equation}
    \langle D,E \rangle = \langle D,E +\divv(f) \rangle.
  \end{equation}
  Using the local-to-global formula (\ref{local_decomposition}) we write
  \begin{equation}
   \langle D,E +\divv(f) \rangle = \sum_\nu \langle D,E+\divv(f) \rangle_{\nu} \, . 
  \end{equation}
  We can now finish by appealing to Proposition~\ref{prop:regularization}.
\end{proof}

\section{Proof of the main theorem} \label{sec:proof_main} 

We start by recalling the notation for our main theorem as stated in the introduction. We have a proper geometrically integral curve $X_0$ defined over a number field $K$ with a single node, a smoothing deformation $\pi \colon X \to S$ of $X_0$ over $K$, and a nonzero cotangent vector $\chi \in \Omega_{S,0}$. For the case where $X_0$ has multiple nodes, see Section~\ref{sec:smooting_multiple_nodes}.

Let $x$ be the node of $X_0$.  Without loss of generality we assume that the tangent directions at $x$ are defined over $K$. 
Let $C$ be the normalization of $X$ and let $p, q \in C(K)$ denote the preimages of $x$ in $C$. 

Let $u,v \in K(C)$ be $K$-rational local coordinates around $p, q$ so that $\dd u|_p \otimes \dd v|_q = \kappa(\chi)$, where $\kappa$ is the small Kodaira--Spencer map~\eqref{eq:canonical_identification}.

Let $\nu$ be any place of $K$. Definition~\ref{def:limit_local_pairing} specializes to give the regularized local pairing 
\begin{equation}  \label{eq:pq_limit}
  \langle p-q,p-q \rangle_{\kappa(\chi),\nu} = \lim_{p',q' \to p,q} \langle p'-q',p-q \rangle_\nu - \log\abs{u(p')v(q')}_\nu 
\end{equation}
for $p', q' $ varying in $C(\C_\nu)$. In particular, the limit exists and is independent of how $p',q'$ approach $p,q$ by Proposition~\ref{prop:regularization}.

We are going to evaluate these regularized local pairings and use the local-to-global formula Proposition~\ref{prop:local_to_global} to conclude the proof.

\subsection{Contribution from the Archimedean places}

When $\sigma \colon K \hookrightarrow \C$ is a complex embedding of $K$, we denote by $L_{\chi,\sigma}$ the limit mixed Hodge structure associated to the cotangent vector $\chi$ obtained after base changing the family $X/S$ to $\C$ along the embedding $\sigma$.

\begin{theorem}\label{thm:regularized_and_biextension_heights} 
  Let $\nu$ be an Archimedean place of the number field $K$. Then we have
\begin{equation}
  \langle p-q,p-q \rangle_{\kappa(\chi),\nu} = \sum_{\sigma \in \nu}\hgt(\LMHS_{\chi,\sigma}) \, .
\end{equation}
Here $\nu$ consists of a single real embedding or is a pair of complex conjugate embeddings. 
\end{theorem}
\begin{proof}
Rewrite~\eqref{eq:pq_limit} using~\eqref{eq:local_pairing_at_infty}. Now the result is just a restatement of Theorem~\ref{height_equals_deresonation}.
\end{proof}

\subsection{Contribution from the non-Archimedean places}

We now determine an explicit formula for the  limit appearing in~\eqref{eq:pq_limit} in the case that the place $\nu$ is non-Archimedean, i.e., equal to a maximal ideal of the ring of integers $\co_K$ of $K$. 

Let $\cc$ be a proper regular flat model of $C$ over $\co_K$. The regular model $\cc$ gives a (canonical) $\co_K$-sublattice inside $\Omega_{C,p}$ and $\Omega_{C,q}$ via $\Omega_{\cc/\co_K,\overline{p}}$ and $\Omega_{\cc/\co_K,\overline{q}}$. We carry the resulting $\co_K$-sublattice of $\Omega_{C,p} \otimes_K \Omega_{C,q} $ to $\Omega_{S,0}$ via the canonical isomorphism $\kappa \colon \Omega_{S,0} \isoto \co_{C,p} \otimes_K \co_{C,q}$ and define a norm $\|\cdot\|$ on $\Omega_{S,0}$ via
\begin{equation}\label{eq:intro_norm}
\log \norm{\chi} = \sum_\xp \val_\xp(\chi) \log \Nm(\xp) \, .
\end{equation}
Here $\Nm(\xp)$ denotes the norm of the maximal ideal $\xp$. 

We choose a vertical $\Q$-divisor $\Phi$ on $\cc$ so that $\overline{p}-\overline{q}+\Phi$ is of degree zero on every geometric component of every special fiber of $\cc$, see Section~\ref{subsec:disjoint}. 

\begin{theorem}\label{thm:local_nonarch_limit}
  Let $\xp$ be a finite prime  of $K$. Then we have
  \begin{equation} \label{eq:thm:local_nonarch_limit}
    \langle p-q,p-q \rangle_{\kappa(\chi),\xp} = \left(\val_\xp \chi + 2\,\iota_\xp(\overline{p},\overline{q}) - \iota_\xp(\overline{p}-\overline{q},\Phi)\right)\log \Nm(\xp).
  \end{equation}
\end{theorem}
\begin{proof}
  We recall from~\eqref{eq:pq_limit} that 
  \begin{equation} \label{eq:local_limit}
    \langle p-q,p-q \rangle_{\kappa(\chi),\xp} = \lim_{p',q' \to p,q} \langle p'-q',p-q \rangle_\xp - \log\abs{u(p')v(q')}_\xp.
  \end{equation}

  Prior to taking the limit, fix two points $p',q' \in C(K_\xp)$ close to but distinct from $p,q$. Let $\overline{p}'$ and $\overline{q}'$ be the closures of $p'$ and $q'$ on the proper regular $\co_{K,\xp}$-model $\cc \otimes \co_{K,\xp}$.  Using the definition of the local non-Archimedean N\'eron pairing in (\ref{def:local_non_arch_pairing}), the expression $\langle p'-q',p-q \rangle_\xp$ can be expanded out using the local intersection products:
  \begin{equation}
    \langle p'-q',p-q \rangle_\xp = -\left( \iota_\xp(\overline{p}',\overline{p}) - \iota_\xp(\overline{p}',\overline{q}) - \iota_{\xp}(\overline{q}',\overline{p}) + \iota_{\xp}(\overline{q}',\overline{q}) + \iota_{\xp}(\overline{p}'-\overline{q}',\Phi)  \right) \log \Nm(\xp).
  \end{equation}
  Let $m = \val_\xp (\dd u|_p)$. Let $\pi$ be a uniformizer at $\xp$. Then  $z \coleq \pi^{-m} u$ has vanishing locus $\overline{p}$ in a neighbourhood of $\overline{p}$ on the local model $\cc \otimes \co_{K,\xp}$. We then have
  \begin{equation}
    \iota_\xp(\overline{p}',\overline{p}) = \val_\xp(z(p')) \, , 
  \end{equation}
  and the $\log\abs{u(p')}_\xp$ term in~\eqref{eq:local_limit} evaluates as
  \begin{equation}
    \log\abs{u(p')}_\xp = \log \abs{\pi^m \cdot z(p')}_\xp = -(\val_\xp ( \dd u|_p)  + \val_\xp(z(p'))) \log \Nm(\xp).
  \end{equation}
The divergent terms $ \pm \val_\xp(z(p')) \log \Nm(\xp)$ cancel one another. A similar line of reasoning goes for the point~$q$. The remaining terms are continuous as $p',q' \to p,q$ so we can evaluate them in the limit. We have $\val_\xp ( \dd u|_p) + \val_\xp ( \dd v|_q) =\val_\xp \chi$. This leads to~(\ref{eq:thm:local_nonarch_limit}).
\end{proof}

\subsection{Final steps of the proof}

We can now put together the Archimedean and non-Archimedean contributions and complete the proof of our main theorem.

First of all, by Proposition~\ref{prop:local_to_global}, we have
\begin{equation} \label{eq:I}
 \hgt(p-q) = \sum_\nu \langle p-q, p-q \rangle_{\kappa(\chi),\nu} \, . 
\end{equation}
By Theorem~\ref{thm:regularized_and_biextension_heights}, we have
\begin{equation} \label{eq:II}
\sum_{\nu | \infty} \langle p-q, p-q \rangle_{\kappa(\chi),\nu} = \sum_{\sigma \colon K \hookrightarrow \C}  \hgt(\LMHS_{\chi,\sigma}) = \hgt(\LMHS_\chi) \, . 
\end{equation}
For any two divisors $\cd,\ce$ on $\cc$ with disjoint generic support we define
\begin{equation}\label{eq:finite_intersection}
(\cd, \ce)_{\mathrm{fin}} \coleq \sum_\xp \iota_\xp(\cd,\ce) \log \Nm(\xp),
\end{equation}
where $\iota_\xp$ refers to the total intersection multiplicity over $\xp$.

From Theorem~\ref{thm:local_nonarch_limit} we obtain
\begin{equation} \label{eq:III}
\begin{split}
\sum_{\nu \nmid \infty} \langle p-q, p-q \rangle_{\kappa(\chi),\nu} & = 
\log \norm{\chi} + 2 \,(\overline{p} \cdot \overline{q})_{\mathrm{fin}} - \left( (\overline{p}-\overline{q})\cdot \Phi \right)_{\mathrm{fin}} \, .
 \end{split}
\end{equation}
The main theorem follows upon combining (\ref{eq:I}), (\ref{eq:II}) and (\ref{eq:III}).

\subsection{The case of multiple nodes} \label{sec:smooting_multiple_nodes}

We would like to make the observation that our techniques and results generalize easily to compute N\'eron-Tate heights of divisors of the form $D=\sum_{i=1}^m (p_i - q_i)$ on $C$ where all points appearing in the expression are distinct. Of course, any divisor \emph{class} of degree zero can be represented in this manner but not every divisor is of this form. 

In this more general setting one would consider the $m$-nodal curve $X_0$ obtained from $C$ by gluing each $p_i$ to $q_i$. 
Fix a smoothing deformation $X/S$ of $X_0$. Let $\LMHS_\chi$ be the limit mixed Hodge structure associated to $X/S$ as $t \to 0$ where $\dd t|_0 = \chi$. 

The twist $B=\LMHS_{\chi}(1)$ of the resulting limit mixed Hodge structure can appropriately be called a \emph{rank $m$-biextension}. Indeed, there are canonical finitely generated free abelian groups $M_1, M_2$ of rank $m$ such that the mixed Hodge structure $B$ has weight graded pieces 
\begin{equation} \Gr^W B =  \left(\Z(1) \otimes M_1 \right)  \oplus H \oplus   M_2  \simeq \Z(1)^m \oplus H \oplus \Z^m \end{equation}
where $H=\H^1(C,\Z(1))$ is a pure Hodge structure of weight $-1$ (compare with Definition~\ref{def:biextension}). The obvious period matrix for $B$, using bases respecting the filtrations and using the notation in Definition~\ref{def:period_matrix}, will have entries $a,b,c$ that are matrices and where the corner entry $1$ becomes the $m\times m$ identity matrix. 
The formula for the height of a period matrix as in Definition~\ref{def:hgt_of_matrix} makes sense but returns a real $m\times m$ matrix. The sum of the entries of this height matrix is ``the total height'' of $\LMHS_\chi$.

We claim that the total height of $\LMHS_\chi$ is once again a regularized Archimedean self-intersection. Let us first introduce the small Kodaira--Spencer maps that we will use. For any $i=1,\dots,m$ the Kodaira--Spencer map gives
\begin{equation}
 \kappa_i \colon \Omega_{S,0} \isoto \Omega_{C,p_i}\otimes \Omega_{C,q_i}.
\end{equation}
Define $\kappa(\chi) \coleq \kappa_1(\chi) \otimes \cdots \otimes \kappa_m(\chi)$ so that we have
\begin{equation}
  \kappa \colon \Omega_{S,0} \isoto \bigotimes_{i=1}^m \Omega_{C,p_i}\otimes \Omega_{C,q_i}.
\end{equation}

The total height of $\LMHS_\chi$ coincides with the $\kappa(\chi)$-regularized Archimedean self-intersection of $D$, as in Definition~\ref{def:limit_local_pairing}. The proof of this fact follows easily from Section~\ref{sec:lmhs} as the key analysis is purely local at each node. The  generalization of our main theorem then follows by a straightforward modification of the proof of Theorem~\ref{thm:local_nonarch_limit} for the $\kappa(\chi)$-regularized non-Archimedean self-intersections. 

The ultimate generalization of the main theorem to general divisors $D= \sum_{i} a_i p_i$ with $\sum_i a_i =0$ seems doable in principle. However note that then $X_0$ must be constructed with more complicated singularities and this, in turn, makes the study of the limit mixed Hodge structure more involved.  We will not investigate this issue further here and leave the details to the interested reader.

\end{document}